\documentclass[12pt]{article}
\usepackage{amssymb,amsmath, amsthm}
\usepackage{color,xcolor}
\newcommand{\eb}{\begin{equation}}
\newcommand{\ee}{\end{equation}}
\newcommand{\ebx}{\begin{equation*}}
\newcommand{\eex}{\end{equation*}}
\newtheorem{lemma}{Lemma}[section]
\newtheorem{proposition}[lemma]{Proposition}
\newtheorem{theorem}[lemma]{Theorem}

\newtheorem{corollary}[lemma]{Corollary}
\newtheorem{definition}[lemma]{Definition}

\allowdisplaybreaks

\makeatletter
\renewcommand*\env@matrix[1][*\c@MaxMatrixCols c]{%
  \hskip -\arraycolsep
  \let\@ifnextchar\new@ifnextchar
  \array{#1}}
\makeatother

\begin{document}

\title{Orthogonal pair and a rigidity problem for Segre maps between hyperquadrics}

\author{Yun Gao\footnote{School of Mathematical Sciences, Shanghai Jiao Tong University, Shanghai, People's Republic of China. \textbf{Email:}~gaoyunmath@sjtu.edu.cn}, }

\maketitle

\begin{abstract}
Being motivated by the orthogonal maps studied in \cite{GN1}, orthogonal pairs  between the projective spaces equipped with  possibly degenerate Hermitian forms were introduced. In addition, orthogonal pairs are generalizations of  holomorphic Segre maps between Segre families of real hyperquadrics. We showed that  non-degenerate holomorphic orthogonal pairs also have certain rigidity properties under a necessary codimension restriction.
As an application, we got a rigidity theorem for  Segre maps related to Heisenberg hypersurfaces obtained by Zhang  in \cite{Zh}.
\end{abstract}

\section{Introduction}
Let $U$ be an open subset in $\mathbb C^n$ and $M$ be a real analytic
hypersurface in $U$ with a defining function $\psi(z,\bar z)$. The
complexification of $M$ is defined as:
 $\mathcal M:=\{(z,\xi)\in U\times  \mbox{Conj}(U): \psi(z,\xi)=0\}$.
  Define the complex analytic varieties $Q_{\xi}:=\{z\in U: \psi(z,\xi)=0\}$
   for $\xi\in \mbox{Conj}(U)$ and  $\hat{Q}_{z}:=\{\xi\in \mbox{Conj}(U): \psi(z,\xi)=0\}$
   for $z\in M$. Then $Q_{\xi}$ and $\widehat{Q}_{z}$ are called  {Segre varieties}
   associated with  $M$. And the complex hypersurface $\mathcal M$ is called
    the { Segre family} of $M$. The term ``family'' comes from the
fact that $\mathcal M$ is foliated by  $\{Q_{\xi}\}$ as $\xi$ varies
over $U$. Segre varieties have been  basic tools in the  study of CR
mappings between real analytic hypersurfaces as originated in the
paper of Webster \cite{We}. (See also [Fa1], [BER], [Hu1], [Hu2],
[Zh],etc, for many related references.)

 Let $\widetilde M$ be also a real analytic hypersurface
 in $\widetilde U\subset\mathbb C^N$ and $\widetilde {\mathcal{M}}$ be its associated
  Segre family.  Let $F$ be a holomorphic map from $\mathcal M$ into
   $\widetilde {\mathcal{M}}$. If for $(z,\xi)\in \mathcal M$,
    there exists $(z',\xi')\in \mathcal M$
    such that $F((\{z\},\hat{Q}_{z}))\subset ( \{z'\},\hat{Q}_{z'})$
    and $F((Q_{\xi},\{\xi\}))\subset (Q_{\xi'},\{\xi'\})$,
    then $F$ is called a \textit{holomorphic Segre map}. It is known that
  $F$ is always of the form
$F(z,\xi)=(f_1(z),f_2(\xi))$ (\cite{Fa2}, (\cite{HJ}, Lemma 6.1)). A basic property  for Segre families is its invariance for
holomorphic maps from $M$ into $\tilde M$.  Namely, a holomorphic
map from $M$ into $\widetilde M$  induces a holomorphic Segre map
from $\mathcal M$ to $\widetilde {\mathcal{M}}$. Segre maps, as well
as  their rigidity problems, have been studied by many authors. In \cite{Zh}, Zhang proved a rigidity theorem for  holomorphic Segre
maps from $\mathcal H^n$ to $ \mathcal H^N$ where $\mathcal H^n$
denotes the Segre family associated to the Heisenberg hypersurface
$\mathbb H^n=\{(z_1,\cdots, z_{n-1},\zeta)\in \mathbb C^n:
Im(\zeta)=\sum_{j=1}^{n-1}z_j\bar z_j\}$.
In
\cite{An2} and \cite{An2}, Angle gave several sufficient conditions insuring that
the Segre maps between the complexification of real analytic
hypersurfaces are Segre transversal or null.
% A  reader can also find
%in \cite{An1}, \cite{An2}, \cite{Zh}, as well as the references
%therein,  recent studies on the rigidity of Segre maps.
% Many people are particularly   interested in the
%  holomorphic Segre map between the Segre families of real hyperquadrics.

%The classical method to study Segre maps is Chern-Moser's normal form theory,
%in which the computations are done by using good coordinates.
In \cite{GN1}, the authors  studied  various rigidity properties for
orthogonal maps that generalize holomorphic proper maps between
hyperquadrics. To be more specific,
%various rigidity properties for
%orthogonal maps rigidity for dealt with the CR maps between
%hyperquadrics using the so-called \textit{orthogonal maps}.
we let $r,s,t\in\mathbb N$ and denote by $\mathbb C^{r,s,t}$ for the
Euclidean space equipped with the standard (possibly degenerate)
Hermitian bilinear form whose eigenvalues are $+1$, $-1$ and $0$
with multiplicities $r$, $s$ and $t$, respectively. Write its
projectivization as $\mathbb P^{r,s,t}:=\mathbb P\mathbb C^{r,s,t}$.
When $t=0$, we use the notation $\mathbb P^{r,s}$ instead of $\mathbb P^{r,s,0}$. %Then the notion of orthogonality on $\mathbb C^{r,s,t}$ (denoted by $\perp$) decends naturally to $\mathbb P^{r,s,t}$. Furthermore, if $M$ is the hyperquadric in $\mathbb P^{n}$ defined by the equation $\{|z_1|^2+\cdots+|z_r|^2=|z_{r+1}|^2+\cdots+|z_{n}|^2\}$, then the Segre variety $Q_z$ is the orthogonal complement of $z$ in $\mathbb P^{r,n-r}$ for any $z\in M$. 
A local holomorphic map $f$ from $U\subset \mathbb P^{r,s,t}$
into $\mathbb P^{r',s',t'}$ is called an local orthogonal map if  $f(z)\perp f(w)$ for any $z$, $w\in U$ such that $z\perp w$.
Motivated by this definition, we then define a local orthogonal pair from $\mathbb P^{r,s,t}$
into $\mathbb P^{r',s',t'}$ to be a pair of  local holomorphic maps
$f_1$, $f_2$ from $U\subset\mathbb P^{r,s,t}$ into $\mathbb
P^{r',s',t'}$ such that $f_1(z)\perp f_2(w)$ whenever $z\perp w$ for
$z,w\in U$. It is easy to find that $f$ is a local orthogonal map from $\mathbb P^{r,s,t}$ to $\mathbb P^{r',s',t'}$ if and only if $(f,f)$ is a local orthogonal pair.
 In addition, a Segre map is easily seen to be a local
orthogonal pair in this sense. Our main rigidity result is the following theorem.

\begin{theorem}\label{main 2}
Suppose that $r'+s'\le 2(r+s)-3$. Then any local orthogonal pair
from $\mathbb P^{r,s}$ to $\mathbb P^{r',s',t'}$ is either null or
quasi-standard.
\end{theorem}

Here,  a local orthogonal pair $F=(f_1,f_2)$ is called a null pair
if $f_1(z)\perp f_2(w)$ for any $(z,w)$ in  domain of definition,
namely the Segre transversality fails at each point.   $F$ is called
quasi-standard if it is  a ``direct sum" of two parts, of which one
comes from a linear isometry pair up to a constant from $\mathbb
C^{r,s}$ into $\mathbb C^{r',s',t'}$ and the other is null. A more
precise definition will be given in Section~\ref{local orthogonal
pair}. The theorem above generalizes the following result of Zhang:

\begin{theorem}[\cite{Zh}]
Suppose $N\leq 2n-2$. Then every local Segre map from $\mathcal H^n$
into $\mathcal H^N$ is either null or quasi-standard.
\end{theorem}

Since $f$ is a local orthogonal map from $\mathbb P^{r,s}$ to $\mathbb P^{r',s',t'}$ if and only if $(f,f)$ is a local orthogonal pair, therefore Theorem \ref{main 2} also generalizes the analogous result for orthogonal maps obtained in \cite{GN1}.

Notice that  Segre families of hyperquadrics of any signature are
Segre isomorphic to each other.  The novelty of our theorem is that
we allow the target space to have degenerate bilinear form. Also,
our proof is different and simpler from the one employed in \cite{Zh}. In addition, we can apply our results to orthogonal maps between generalized balls.

Our method of the proof is based on an analysis of  linear subspaces
theory initiated in the paper of Faran(\cite{Fa1}).
 First, the orthogonal
complements (which are well defined on $\mathbb P^{r,s,t}$) are
linear subspaces that are reserved by local orthogonal pairs. Then
the dimension of the linear span of the image of any general
$k$-dimensional linear subspace under $f_1$ controls the
corresponding dimension of the image of its orthogonal complement
under $f_2$ for  $1< k< r+s+t$ and vice versa. This mutual
restriction  between $f_1$ and $f_2$ together with the dimensional
formula obtained in~\cite{GN2} (see Lemma \ref{analysis lemma}) that
is used to estimate the size of the linear spans of the image of a
linear subspace in any dimension, forces the linearity (or
quasi-linearity) of the orthogonal pair under the stated assumptions
for $r,s,r',s'$.

%-------------------------------------------------------------------------------------------

\section{Notations and conventions}
For  non-negative integers $r,s,t$, consider the standard Hermitian
form of signature $(r;s;t)$ on $\mathbb C^{r+s+t}$ whose the
eigenvalues are $1$, $-1$ and $0$ with multiplicities $r$, $s$ and
$t$, respectively, which is represented by the matrix
 $H_{r,s,t}=\begin{pmatrix}
 I_r&0&0\\
 0&-I_s&0\\
 0&0&0_t
 \end{pmatrix}$ under the standard coordinates. Let $n=r+s+t$.  Define the (possibly degenerate) indefinite inner product of signature $(r;s;t)$ on $\mathbb C^{n}$ as
$$
    \langle z, w\rangle_{r,s,t}
    =z_1\bar w_1+\cdots+z_r\bar w_r-z_{r+1}\bar w_{r+1}-\cdots - z_{r+s}\bar w_{r+s},
$$
where $z=[z_1,\ldots,z_{n}]$ and $w=[w_1,\ldots,w_{n}]$. We denote
such a $\mathbb C^{n}$ with the aforementioned  Hermitian inner
product by $\mathbb C^{r,s,t}$ and write its projectivization by
$\mathbb P^{r,s,t}:=\mathbb P\mathbb C^{r,s,t}$. If $\langle z,
w\rangle_{r,s,t}=0$, we say that $z$ is orthogonal to $w$ and write
$z\perp w$. When $z\perp z$, we say that $z$ is a \textit{null
point}.  In addition, we also write the \textit{orthogonal
complement} of $z$ as
$$z^{\perp}=\{w\in  \mathbb C^{r,s,t} \mid \langle z, w\rangle_{r,s,t}=0\}.$$
The inner product $\langle z, w\rangle_{r,s,t}$ does not descend to
$\mathbb P^{r,s,t}$. However, the orthogonality and the orthogonal
complement are well defined over $\mathbb P^{r,s,t}$.

More generally,
let $V$ be an $n$-dimensional complex vector space equipped with an abstract
Hermitian inner product $H_V$ (possibly degenerate or indefinite)
of signature $(r;s;t)$, where $n=r+s+t$. Let
 $\mathbb PV$ be its projectivization.
  The notion of orthogonality can be similarly defined on $\mathbb PV$.
  In addition, a linear isometric isomorphism $F:\mathbb C^{r,s,t}\rightarrow V$
   induces a biholomorphic map $\widetilde F:\mathbb P^{r,s,t}\rightarrow\mathbb PV$.
    For our purpose here, it is sufficient for us to simply identify any
such projective space $\mathbb PV$ with $\mathbb P^{r,s,t}$.
% through
%any such biholomorphism and we write $\mathbb PV\cong\mathbb
%P^{r,s,t}$ for such identification.

Now let $H$ be a $k$-dimensional complex vector subspace
 in $\mathbb C^{r,s,t}$ and suppose the restriction of
 $H_{r,s,t}$ on $H$ is of signature $(a; b; c)$.
 Then $\mathbb PH \cong \mathbb P^{a,b,c}$.
 We call $\mathbb P H$ an\textit{ $(a,b,c)$-subspace}.
 (When $c=0$, we will simply call $\mathbb PH$ an $(a,b)$-subspace.)
 For brevity,  we denote an $(a,b,c)$-subspace by $H^{a,b,c}$.
 Clearly,
 we have $0\le a\le r$, $0\le b\le s$, $0\le c\le \min\{r-a,s-b\}+t$ and $a+b+c=k$.

In this article, when we say a \textit{general point}, \textit{general line}, etc., we mean our choice of the point or line, etc. is anything outside some finite union of complex analytic subvarieties.

%-------------------------------------------------------------------------------------------

\section{Local orthogonal pair}\label{local orthogonal pair}

For real hyperquadrics in a complex projective space, we note that
Segre maps preserve  Segre varieties which are the orthogonal
complements to points with respect to the orthogonality described
before. Motivated by this property, we  give the following
definition:

\begin{definition}\label{orth}
Let $U\subset\mathbb P^{r,s,t}$ be a connected open set containing a null point
 and let
$f_1$, $f_2$ be two holomorphic maps from $U$ into $\mathbb
P^{r',s',t'}$. We call $F:=(f_1,f_2)$   an  {orthogonal pair} if
$f_1(z)\perp f_2(w)$ for any $z$,$w\in U$ with $z \perp w$. We also
call $F$ a local orthogonal pair from  $\mathbb P^{r,s,t}$ to
$\mathbb P^{r',s',t'}$ .
\end{definition}

 We remark that when  $U$ does not contain any null point,
 the orthogonality condition in the above definition  could be vacuous, for
  there might  not be any pair of orthogonal points in $U$.

In what follows,  $f_1$ or $f_2$ is said to  map a line (or, a
$k$-plane) to a line (respectively, a $k'$-plane) if it maps the
intersection of a line with $U$ (or, a $k$-plane with $U$) into a
line (or  a $k'$-plane, respectively).

For brevity,  we also give  the following definitions.

\begin{definition}\label{iso}
Let $\sigma$ and $\tau$ be two linear maps from $\mathbb C^{r,s,t}$ to $\mathbb
 C^{r',s',t'}$.
 The pair $(\sigma, \tau)$ is called an {conformal pair}
 from $\mathbb C^{r,s,t}$ to $\mathbb
 C^{r',s',t'}$ if there exists  a non-zero $\lambda\in\mathbb C$
  such that $\langle \sigma(z),\tau(w)\rangle_{r',s',t'}=\lambda\langle z,w \rangle_{r,s,t}$
  for any $z, w\in \mathbb C^{r,s,t}$.
\end{definition}

\begin{definition}\label{stand}
Let $F=(f_1,f_2)$ be a  local orthogonal pair from $\mathbb
P^{r,s,t}$ to $\mathbb P^{r',s',t'}$. We call $F$ a {standard pair}
if it is induced by a conformal pair from $\mathbb C^{r,s,t}$ to
$\mathbb C^{r',s',t'}$. We call call $F$ a {null pair} if
$f_1(z)\perp f_2(w)$ for any $z,w $.
\end{definition}

For a non-trivial orthogonal direct sum decomposition $\mathbb
C^{r,s,t}=A\oplus B$, there are two standard projections
$\pi^A:\mathbb P^{r,s,t}\dasharrow\mathbb PA $ and $\pi^B: \mathbb
P^{r,s,t} \dasharrow\mathbb PB$ (as rational maps). For a pair of
holomorphic maps $F=(f_1,f_2)$, we will write $(\pi^A\circ
f_1,\pi^A\circ f_2)$ for $(\pi^A,\pi^A)\circ F$ and write similarly
$(\pi^B,\pi^B)\circ F$  (whenever the compositions are defined as
rational maps). By the same token, for a proper vector subspace
$S\subset\mathbb C^{r',s',t'}$ such that the restriction of the
Hermitian form on $S$ is non-degenerate, we can define unambiguously
the projection $\pi^S:\mathbb P^{r',s',t'}\dasharrow\mathbb PS$, as
$\mathbb C^{r',s',t'}=S\oplus S^\perp$. Thus, we can also define
$(\pi^S,\pi^S)\circ F$ in this case.

An orthogonal pair interacts naturally with an orthogonal decomposition:

\begin{lemma}\label{projection 2}
Let $F=(f_1,f_2)$ be a local orthogonal pair from $\mathbb
P^{r,s,t}$ to $\mathbb P^{r',s',t'}$ and let $\mathbb
C^{r',s',t'}=A\oplus B$ be an orthogonal decomposition.
Assume that the union of the  images of $f_1$ and $f_2$ is not
contained in $\mathbb PA\cup\mathbb PB$. If $(\pi^A,\pi^A)\circ F$
is a local orthogonal pair, then so is $(\pi^B,\pi^B)\circ F$. In
addition, if $(\pi^A,\pi^A)\circ F$ is null, then
$(\pi^B,\pi^B)\circ F$ is null if and only if $F$ is null.
\end{lemma}
\begin{proof}
Let $U$ be the  domain of definition for $F$.  Let $V_1:=U\setminus
f_1^{-1}(\mathbb PA\cup\mathbb PB)$, $V_2:=U\setminus
f_2^{-1}(\mathbb PA\cup\mathbb PB)$ and $V=V_1\cap V_2$. Let $p,q\in
V$ be two orthogonal points and $p'=f_1(p)$, $q'=f_2(q)$.

Let $\ell_{p'}$, $\ell_{q'}$, $\ell^A_{p'}$, $\ell^A_{q'}$,
 $\ell^B_{p'}$ and $\ell^B_{q'}$ be the lines in $\mathbb C^{r',s',t'}$
  corresponding to $p'$, $q'$, $\pi^A(p')$, $\pi^A(q')$, $\pi^B(p')$ and $\pi^B(q')$,
   respectively. Then $\ell_{p'}^A\perp\ell_{q'}^B$ and $\ell_{q'}^A\perp\ell_{p'}^B$ by definition, and by the orthogonality of $F$ and $(\pi^A,\pi^A)\circ F$, we also have $\ell_{p'}\perp\ell_{q'}$ and $\ell^A_{p'}\perp\ell^A_{q'}$.

For  non-zero vectors $v_{p'}\in\ell_{p'}$ and $v_{q'}\in\ell_{q'}$,
write $v_{p'}=v_{p'}^A+v_{p'}^B$ and $v_{q'}=v_{q'}^A+v_{q'}^B$ for
the decompositions with respect to $\mathbb C^{r',s',t'}=A\oplus B$.
Since $p', q'\not\in\mathbb PA\cup\mathbb PB$, it follows that
$v_{p'}^A, v_{p'}^B, v_{q'}^A, v_{q'}^B$ are non-zero. From the
previous paragraph, we know that $v_{p'}\perp v_{q'}$,
$v_{p'}^A\perp v_{q'}^B$,  $v_{p'}^B\perp v_{q'}^A$,  and
$v^A_{p'}\perp v^A_{q'}$. Thus, we have $v^B_{p'}\perp v^B_{q'}$ and
thus $\ell^B_{p'}\perp\ell^B_{q'}$. Hence $(\pi_B,\pi_B)\circ F$ is
an orthogonal pair.

If $(\pi^A,\pi^A)\circ F$ is null, then $v^A_{p'}\perp v^A_{q'}$ for
any $p,q\in V$ (not just orthogonal pair $(p,q)$).
 Therefore, for any $p,q\in V$, we have $v^B_{p'}\perp v^B_{q'}$
  if and only if $v_{p'}\perp v_{q'}$.
\end{proof}

Applying a similar argument, one  also  obtains the following lemma:

\begin{lemma}\label{projection 1}
Let $F:=(f_1,f_2)$ be a local orthogonal pair from $\mathbb P^{r,s,t}$ to $\mathbb P^{r',s',t'}$. If the image of $f_1$ or $f_2$ is contained in an $(a,b)$-subspace $H^{a,b}\subset\mathbb P^{r',s',t'}$, then either $F$ is null or $(\pi,\pi)\circ F$ is a non-null local orthogonal pair, where $\pi$ is the projection from $\mathbb P^{r',s',t'}$ to $H^{a,b}$.
\end{lemma}

We now make the following definition:

\begin{definition}
Let $F:=(f_1,f_2)$ be a local orthogonal pair from $\mathbb
P^{r,s,t}$ into $\mathbb P^{r',s',t'}$. We call $F$ {quasi-standard}
if one of the following scenarios occur:

(1) There exists a vector subspace $A\subset \mathbb C^{r',s',t'}$ such that the image of $f_1$ or $f_2$ is contained in $\mathbb PA$ and $(\pi^A,\pi^A)\circ F$ is standard.

(2) There exists a non-trivial orthogonal decomposition $\mathbb C^{r',s',t'}=A\oplus B$ such that $(\pi^A,\pi^A)\circ F$ is standard and $(\pi^B,\pi^B)\circ F$ is  null.
\end{definition}

\noindent{\textbf{Example.}} If we split the homogeneous coordinates on
$\mathbb P^{r,s}$ as $[z_1,\ldots,z_{r+s}]=[z^+,z^-]$,
where $z^+=[z_1,\cdots, z_r]$, $z^-=[z_{r+1},\cdots, z_{r+s}]$ and let $f_1$ $f_2$ be
 the two holomorphic maps from $\mathbb P^{r,s}$ to $\mathbb P^{r',s'}$
 (with $r<r'$ and $s<s'$) defined by $f_1[z^+,z^-]=[z^+,0,z^-,0]$
  and  $f_2[w^+,w^-]=[\phi[w]w^+,\phi_1[w],\phi[w]w^-,\phi_2[w]]$,
  where $\phi$, $\phi_1$, $\phi_2$ are homogeneous polynomials of the same degree,
   then $(f_1,f_2)$ is quasi-standard.

Write the standard $(r',s')$-subspace of $\mathbb P^{r',s',t'}$ as $\mathbb P^{r',s'}\subset\mathbb P^{r',s',t'}$. From the following lemma, a local orthogonal pair $F$ from $\mathbb P^{r,s,t}$ to $\mathbb P^{r',s',t'}$ naturally gives rise to a local orthogonal pair from $\mathbb P^{r,s,t}$ to $\mathbb P^{r',s'}$ unless the image of $f_1$ or $f_2$ lies entirely in
$(\mathbb P^{r',s'})^\perp\cong\mathbb P^{0,0,t'}$, the set of indeterminacy for the projection.

\begin{lemma}\label{projection}
Let $F$ be a local orthogonal pair from $\mathbb P^{r,s,t}$ to $\mathbb P^{r',s',t'}$ and
 $\pi$, $\pi^{\perp}$ be the projections from $\mathbb P^{r',s',t'}$ to $\mathbb P^{r',s'}$ and $\mathbb P^{0,0,t'}$ respectively. Then, $(\pi,\pi)$ is standard, $(\pi,\pi)\circ F$ is orthogonal and $(\pi^{\perp},\pi^{\perp})\circ F$ is null whenever the compositions are defined. In addition, $F$ is null if and only if $(\pi,\pi)\circ F$ is null.
\end{lemma}
\begin{proof} The proof follows by applying the above lemma to
 the projection $\pi:\mathbb C^{r',s',t'}\rightarrow\mathbb
C^{r',s'}$.
\end{proof}

\begin{definition}
Let $F:=(f_1,f_2)$ be a local orthogonal pair from $\mathbb P^{r,s,t}$ to $\mathbb P^{r',s',t'}$. If the dimension of the linear span of the image of $\pi\circ f_1$ or $\pi\circ f_2$ is less than $\dim(\mathbb P^{r,s})$, where $\pi$ is the projection from $\mathbb P^{r',s',t'}$ to $\mathbb P^{r',s'}$, then we say that $F$ is \textbf{degenerate.}
\end{definition}

\begin{proposition}\label{null}
Every degenerate local orthogonal pair is null.
\end{proposition}
\begin{proof}
Let $F:=(f_1,f_2)$ be an orthogonal pair from a connected open
 set $U\subset\mathbb P^{r,s}$ into $\mathbb P^{r',s',t'}$.

First, we consider the case $t'=0$.
 Without loss of generality, we assume that the linear span of $f_1(U)$ is a $d$-dimensional linear subspace $\Sigma\subset\mathbb P^{r',s'}$, where $d<\dim(\mathbb P^{r,s})=r+s-1$.
 Let $d'$ be the dimension of the linear span of $f_1(H\cap U)$ for a
  general hyperplane $H\subset\mathbb P^{r,s}$. Then $d'\le d< r+s-1$.

 If $d'=d$, then for a general point $\alpha\in U$ such that its orthogonal
  complement is a hyperplane $H$ intersecting $U$,
   the linear span of $f_1(H\cap U)$ is $\Sigma$ and hence $f_2(\alpha)$
   lies in $\Sigma^{\perp}$. Since
   this holds  for each point in a neighborhood of $\alpha$,
   we conclude that $F$ is a null pair by uniqueness of holomorphic functions.

 If $d'<d$, then $d'<\dim(H)$. From Lemma~\ref{analysis lemma}($ii$),
  the linear span of $f_1(U)$ is contained in a $d'$-dimensional linear subspace.
This  contradicts to the assumption that $d$ is the dimension of the
linear span of $f_1(U)$.

Hence we see that when  $t'=0$ and $F$ is degenerate, $F$ is a null
pair.

 When $t'>0$, by Lemma~\ref{projection}, $(\pi,\pi)\circ F$ is a local orthogonal pair
 and $(\pi^{\perp},\pi^{\perp})\circ F$ is null, where $\pi$ and $\pi^{\perp}$ are
 the projections from $\mathbb P^{r',s',t'}$ to $\mathbb P^{r',s'}$ or $\mathbb P^{0,0,t'}$
 respectively.  And $(\pi,\pi)\circ F$ is a null pair from the argument above.
  Therefore $F$ is a null pair.
\end{proof}

%\noindent\textbf{Remark.} The lemma above is optimal in the following sense.
%If the linear span of $f_1(U)$ is a $(0,0,c)$-subspace in $\mathbb P^{r',s'}$. The condition $c-1<\dim{\mathbb P^{r,s}}$ in the above lemma is optimal.
%Let $f_1[z^+,z^-]=[z^+,z^-,0,z^+,z^-]$ and $f_2[w^+,w^-]=[w^+,-w^-,0,0,0]$ be two holomorphic maps from $\mathbb P^{r,s}$ to $\mathbb P^{r',s'}$. It is clear that $(f_1,f_2)$ is an orthogonal pair and the image of $f_1$ is contain in a $(0,0,r+s)$-subspace $N\subset\mathbb P^{r',s'}$ with $\dim(N)=\dim(\mathbb P^{r,s})$. However, $(f_1,f_2)$ is not a null pair.

From the proposition above, it is easy to get the following result.

\begin{corollary}\label{less}
Let $F:=(f_1,f_2)$ be a local orthogonal pair from $\mathbb P^{r,s}$ to $\mathbb P^{r',s',t'}$.  If $r+s>r'+s'$, then $F$ is null.
\end{corollary}

The following result provides a tool for us to carry out a reduction
argument.

\begin{lemma}\label{projection 3}
Let $F$ be an orthogonal pair from a connected open set $U\subset\mathbb P^{r,s}$ to $\mathbb P^{r',s'}$. Suppose the linear span of  $f_1(U)$ is an $(a,b,c)$-subspace $H^{a,b,c}\subset\mathbb P^{r',s'}$ such that $(a,b)\neq (0,0)$ and $\dim(H^{a,b,c})\le 2\dim(\mathbb P^{r,s})-2$. Then, either $F$ is null or $(\pi,\pi)\circ F$ is a non-null local orthogonal pair, where $\pi$ is the projection from $\mathbb P^{r',s'}$ to any $(a,b)$-subspace $H^{a,b}$ of $H^{a,b,c}$.
\end{lemma}
\begin{proof}
Let $ H^{a+c,b+c}$ be an $(a+c,b+c)$-subspace in $\mathbb P^{r',s'}$
containing $H^{a,b,c}$. From Lemma~\ref{projection 1}, we may, after
replacing $F$ by its composition with the projection to $
H^{a+c,b+c}$, assume that $\mathbb P^{r',s'} =\mathbb P^{a+c,b+c}$.
Fix an $(a,b)$-subspace $H^{a,b}\subset H^{a,b,c}$ and let $\pi$ and
$\pi^{\perp}$ be the projections from $\mathbb P^{a+c,b+c}$ to
$H^{a,b}$ and $(H^{a,b})^\perp$, respectively.

Let $U'\subset U$ be the open subset consisting of $\beta's$ such
that $\beta^\perp\cap U\neq\varnothing$.
 By
Definition~\ref{orth}, $U'$ is non-empty. For $\beta\in U'$, let
$d(\beta)$ be the dimension of the linear span of
$f_1(U\cap\beta^\perp)$. If $d(\beta)=\dim (H^{a,b,c})$, then the
linear span of $f_1(U\cap\beta^\perp)$ is exactly $H^{a,b,c}$. Thus,
$f_2(\beta)\in (H^{a,b,c})^\perp$ by the orthogonality of $F$. If it
is true for a general choice of $\beta$, then we immediately see
that $F$ is null.

If for each $\beta\in U'$, we have $d(\beta)\leq \dim
(H^{a,b,c})-2$, then $d(\beta)\leq 2\dim(\mathbb P^{r,s})-4\leq
2\dim(\beta^\perp)-2$. Since any hyperplane in $\mathbb P^{r,s}$ is
the orthogonal complement of some point, we deduce by
Lemma~\ref{analysis lemma}($i$) that the dimension of the linear
span of $f_1(U)$ is $d+1\leq \dim (H^{a,b,c})-1$, contradicting our
definition of $H^{a,b,c}$.

It remains to handle the case where $d(\beta)=\dim (H^{a,b,c})-1$ for
 a general $\beta\in U'$. In this case, if we let $H(\beta)$ be the
 linear span of $f_1(U\cap\beta^\perp)$, then $H(\beta)$ is a hyperplane in $H^{a,b,c}$.
  We claim that $H(\beta)^\perp\subset H^{a,b,c}$. Assuming
  the claim for the moment, then as $f_2(\beta)\in H(\beta)^\perp\subset H^{a,b,c}$ for a general $\beta$, which implies that $H^{a,b,c}$ contains the image of $f_2$, we see that $F$ is a local orthogonal pair from $\mathbb P^{r,s}$ to $H^{a,b,c}$. The desired result now follows Lemma~\ref{projection}.

To prove the claim,  let $Z:=(H^{a,b,c})^\perp$. Since the ambient space is $\mathbb P^{a+c,b+c}$, it follows easily that $Z\subset H^{a,b,c}$ and $Z$ is indeed just the null space in $H^{a,b,c}$ corresponding to the zero eigenvalue. As $H(\beta)$ is a hyperplane in $H^{a,b,c}$, we have $H(\beta)^\perp\supset (H^{a,b,c})^\perp=Z$ and
$\dim(H(\beta)^\perp)\leq\dim(Z)+1.$ On the other hand,
Let $\hat H^{a,b,c}\subset\mathbb C^{a+c,b+c}$ be the
Hermitian vector subspace of signature $(a,b,c)$ such that
 $\mathbb P\hat H^{a,b,c}=H^{a,b,c}$ and similar notation  for $\hat Z$, $\hat H(\beta)$.
 Now $\hat H(\beta)$ is a hyperplane in $\hat H^{a,b,c}$ and thus its
 image under the canonical projection to the quotient space $\hat H^{a,b,c}/\hat Z$
 is either full or of codimension 1. This implies that $\hat H(\beta)^\perp=
 (\hat H^{a,b,c})^\perp=\hat Z$ and $\hat H(\beta)^\perp\varsupsetneq\hat Z$
  with $\dim(\hat H(\beta)^\perp\cap\hat H^{a,b,c})=\dim(\hat Z)+1$. Combining these
   with the inequality $\dim(H(\beta)^\perp)\leq\dim(Z)+1$ obtained earlier,
   it follows that we always have $\hat H(\beta)^\perp\subset\hat H^{a,b,c}$
   and hence the claim follows.
\end{proof}

\noindent{\textbf{Remark.}} When $\dim(H^{a,b,c})>2\dim(\mathbb P^{r,s})-2$, the conclusion of Lemma~\ref{projection 3} may not hold. For example, let $f_1$ and $f_2$ be the holomorphic maps from $\mathbb P^{2,2}$ to $\mathbb P^{3,4}$ defined as follows
$$f_1[z_1,z_2,z_3,z_4]=[z_1^2,z_2^2,z_1z_2,z_3^2,z_4^2,\sqrt 2z_3z_4,z_1z_2]$$
$$f_2[w_1,w_2,w_3,w_4]=[w_1^2,w_2^2,w_1w_2,w_3^2,w_4^2,\sqrt 2w_3w_4,-w_1w_2].$$
Then $(f_1,f_2)$ is an orthogonal pair and the linear span of $f_1(\mathbb P^{2,2})$ is an $(2,3,1)$-subspace $H^{2,3,1}$ in $\mathbb P^{3,4}$. But $(\pi\circ f_1,\pi\circ f_2)$ is not an orthogonal pair, where $\pi$ is the projection from $\mathbb P^{3,4}$ to $H^{2,3}$.

%We can say a little bit more about the case that one holomorphic map of the local orthogonal pair is linear.

\begin{lemma}\label{linear}
Let $F=(f_1,f_2)$ be a local orthogonal pair from $\mathbb P^{r,s}$ to $\mathbb P^{r',s'}$ and suppose $\dim(\mathbb P^{r,s})=\dim(\mathbb P^{r',s'})$. If $f_1$ or $f_2$ is linear, then $F$ is either null or standard.
\end{lemma}
\begin{proof}
Suppose $F$ is not null and $f_1$ is linear. Then by Proposition~\ref{null}, $F$ is
 non-degenerate. Thus, $f_1$ is a linear biholomorphism. By the orthogonality of $F$,
  it follows immediately that $f_2$ maps lines to lines and hence it is also a
  linear biholomorphism since $F$ is non-degenerate (Lemma~\ref{analysis lemma}($ii$)).
   Let $A_1$ and $A_2$ be two matrices representing $f_1$ and $f_2$ respectively under
    the standard homogeneous coordinates of $\mathbb P^{r,s}$ and $\mathbb P^{r',s'}$.
    By the orthogonality of $F$, for any $u,v\in\mathbb C^{r,s}$, we have
    $\bar u^t\bar A_1^tH_{r',s'}A_2v=0$ if and only if $\bar u^tH_{r,s}v=0$, where $H_{r',s'}$ and $H_{r,s}$ are the Hermitian matrices representing the inner products in $\mathbb C^{r',s'}$ and $\mathbb C^{r,s}$. This implies that $\bar A_1^tH_{r',s'}A_2=\lambda H_{r,s}$ for some non-zero $\lambda\in\mathbb C$ and hence $F$ is standard.
\end{proof}

We are now ready to prove our main theorem.

\begin{theorem}\label{main}
If $r'+s'\le 2\dim(\mathbb P^{r,s})-1$, every local orthogonal pair from $\mathbb P^{r,s}$ to $\mathbb P^{r',s',t'}$ is either null or quasi-standard.
\end{theorem}

\begin{proof}
As in Proposition~\ref{null}, we only need to prove the result for the
 case $t'=0$.
  Let $F=(f_1,f_2)$ be a local orthogonal pair from $U\subset\mathbb P^{r,s}$ into
   $\mathbb P^{r',s'}$ and $F$ is not null. From Proposition~\ref{null}, $F$ is non-degenerate.

The hypothesis is equivalent to that $\dim(\mathbb P^{r',s'})\leq
2\dim(\mathbb P^{r,s})-2$ and so everything is trivial unless
$\dim(\mathbb P^{r,s})\geq 2$.  We will consider the case
$\dim(\mathbb P^{r,s})\ge 3$ first and handle the two dimensional
case later.

If $f_1$ is not linear, then the dimension (denoted by $d$) of the linear span
$S(L)$ of $f_1(U\cap L)$ for a general line $L\subset\mathbb P^{r,s}$ intersecting $U$
 is at least 2. Note that $\dim(L^\perp)=\dim(\mathbb P^{r,s})-2=r+s-3\geq 1$
 and $\dim(S(L)^\perp)=\dim(\mathbb P^{r's'})-d-1=r'+s'-d-2$ for a general line $L$.
 In addition, since any $(r+s-3)$-plane in $\mathbb P^{r,s}$ is the orthogonal
  complement of some line (for the Hermitian form is non-degenerate), we see that $f_2$
  maps $(r+s-3)$-planes into $(r'+s'-d-2)$-planes by the orthogonality of $F$. As
$$
    r'+s'-d-2\leq r'+s'-4\leq 2(r+s)-7=2(r+s-3)-1,
$$
Lemma~\ref{analysis lemma}($i$) implies that the image of $f_2$ is
contained
 in a proper linear subspace $H\subset\mathbb P^{r',s'}$ such
 that $\dim(H)\leq r'+s'-d\leq\dim(\mathbb P^{r',s'})-1$.

Suppose $H\cong\mathbb P^{a,b,c}$ for some $a,b,c$.
 Then $c\leq \min\{r',s'\}\leq\dim(\mathbb P^{r,s})-1$
  since $r'+s'\le 2\dim(\mathbb P^{r,s})-1$. From this we see that
   $(a,b)\neq (0,0)$ since otherwise it would contradict the fact that $F$ is non-degenerate.
As $\dim(H)\leq \dim(\mathbb P^{r',s'})-1\leq 2\dim(\mathbb P^{r,s})-3$, by Lemma~\ref{projection 3} $F_K:=(\pi^K,\pi^K)\circ F$ is a non-null local orthogonal pair, where $\pi^K: \mathbb P^{r',s'}\dasharrow K$ is the projection to any choice of $(a,b)$-subspace $K\subset H$.

The hypotheses of the theorem still hold for $F_K$. We can therefore
 repeat the arguments and deduce inductively that
 there exists some $(r'',s'')$-subspace $\Phi\subset\mathbb P^{r',s'}$
  with projection $\pi_\Phi:\mathbb P^{r',s',t'} \dasharrow\Phi$,
  such that $F_\Phi:=(\pi^\Phi\circ {f_1},\pi^\Phi\circ {f_2})$ is a non-null
   local orthogonal pair of which at least one component, say $\pi^\Phi\circ {f_1}$,
   is linear.  Combining with Proposition~\ref{null} and Lemma~\ref{projection 3}, we may further assume that $\dim(\Phi)=\dim(\mathbb P^{r,s})$ and $\pi^\Phi\circ {f_1}$ is a linear biholomorphism.

Here we note that the previous conclusion is also true when $\dim(\mathbb P^{r,s})=2$ (from which the hypothesis implies $\dim(\mathbb P^{r',s'}) \le 2$). It is because in this case if we replace a general line $L$ in the argument above by a general point, we see immediately
by orthogonality that $f_2$ maps lines to lines and hence by non-degeneracy $f_2$ is a linear biholomorphism onto $\mathbb P^{r',s'}$.

Lemma~\ref{linear} now implies that $F_\Phi$ is standard (and $F$ itself is standard when $\dim(\mathbb P^{r,s})=2$). If $\dim(\mathbb P^{r,s})\geq 3$, we have $\dim(\Phi^\perp)\leq\dim(\mathbb P^{r,s})-3$ since $\dim(\mathbb P^{r',s'})\leq 2\dim(\mathbb P^{r,s})-2$ by hypothesis. Hence, if the images of $f_1$ and $f_2$ are not contained in $\Phi$, then $(\pi^{\Phi^\perp}, \pi^{\Phi^\perp})\circ F$ is defined and null by Corollary~\ref{less}. Therefore, $F$ is quasi-standard.
\end{proof}

\section{Segre map}

In \cite{Zh}, Zhang considered the holomorphic Segre maps from the Segre family $\mathcal H^n$ of the Heisenberg hypersurface in $\mathbb C^n$ defined by $\mathbb H^n=\{(z_1,\cdots, z_{n-1},\zeta)\in \mathbb C^n: Im(\zeta)=\sum_{j=1}^{n-1}z_j\bar z_j\}$. It is well-known that under the Cayley transform, $\mathbb H^n$ is the intersection of the set of null points in $\mathbb P^{1,n}$ with an affine coordinate chart.

\begin{lemma}\label{Segre}
A local holomorphic Segre map from $\mathcal H^n$ into $\mathcal H^N$ is a local orthogonal pair from $\mathbb P^{1,n}$ to $\mathbb P^{1,N}$.
\end{lemma}
\begin{proof}
The Segre family $\mathcal H^n$ is holomorphically foliated by $Q_{\xi}\times \{\xi\}$ and $z\times \hat Q_{z}$, where
$$Q_{\xi}\times \{\xi\}\cong \{z\in \mathbb P^{1,n}: \langle z,\xi\rangle_{1,n}=0\}$$
$$z\times \hat Q_{z}\cong \{\xi \in \mathbb P^{1,n}: \langle z,\xi\rangle_{1,n}=0\}.$$
A local holomorphic Segre map $F=(f_1,f_2)$ from $\mathcal H^n$ into $\mathcal H^N$ satisfies $f_1(
Q_\xi)\subset \hat Q_{f_2(\xi)}$ and $f_2(
\hat Q_z)\subset Q_{f_1(z)}$.

Hence, $f_1$ maps $\xi^{\perp}$ into $(f_2(\xi))^{\perp}$ and $f_2$
maps $z^{\perp}$ into $(f_2(z))^{\perp}$.  Therefore $$\langle
f_1(z),f_2(w)\rangle_{1,N}=0$$ for any $z$, $w\in \mathbb P^{1,n}$
such that $\langle z,w\rangle_{1,n}=0$. Hence $F$ is a local
orthogonal pair from $\mathbb P^{1,n}$ to $\mathbb P^{1,N}$.
\end{proof}

The following rigidity of the Segre maps from $\mathcal H^n$ into $\mathcal H^N$ obtained by Zhang in~\cite{Zh} is a special case of our Theorem~\ref{main}.

\begin{theorem}
Suppose $N\leq 2n-2$. Then every local Segre map from $\mathcal H^n$ into $\mathcal H^N$ is either null or quasi-standard.
\end{theorem}
\begin{proof}
It follows from Theorem~\ref{main} and Lemma~\ref{Segre}.
\end{proof}

\section{Appendix}

The following result about the estimate of the dimension of the holomorphic mappings is the special case of main theorem in \cite{GN2} which is proved by using Green hyperplane restriction Theorem. In order to make the paper more self-contained, we give an analytic proof here, which is inspired by \cite{Fa1}.

\begin{lemma}\label{analysis lemma}
 Let $U\subset\mathbb P^m$ be a connected open set and $F:U\rightarrow\mathbb P^{m'}$ be a holomorphic map. Suppose $F$ maps $l$-planes to $l'$-planes, where $1\leq l \leq m-1$.

 $(i).$ If $l'\leq 2l-1$, then $F$ maps $(l+k)$-planes into $(l'+k)$-planes for $k\geq 0$.

$(ii).$ If $l'\leq l$, then either $F$ is linear or the image of $F$ is contained in an $l'$-plane.

\end{lemma}

\begin{proof}
At first, we will prove that $F$ maps $(l+1)$-planes into $(l'+1)$-planes. Let $q$ be the maximum of the rank of the Jacobian matrix $DF$ of $F$. We consider two cases:

\textbf{(Case I: $q\geq l+1$)} In this case, by restricting to a general $q$-dimensional linear subspace (i.e. those with tangent spaces intersecting trivially with the kernel of $DF$ at some point), we get a map (still denoted by $F$) from $\mathbb P^q$ to $\mathbb P^{m'}$ mapping $l$-planes into $l'$-planes and such that at a general point, the Jacobian matrix of $F$ is of rank $q$. It suffices to show that this $F$ maps $(l+1)$-planes into $(l'+1)$-planes since it then follows that a general (and hence every) $(l+1)$-plane of $\mathbb P^m$ is mapped by (the original) $F$ into an $(l'+1)$-plane.

Suppose on the contrary the linear span of the image of a general $(l+1)$-plane $E\subset\mathbb P^q$ is of dimension at least $(l'+2)$.

Then, for a general point in $x\in E$, there exist multi-indices $\alpha_1,\cdots,\alpha_{l'+1-l} $ and holomorphic tangent vectors $v_1,\cdots, v_{l'+1-l}$ in $T_xE$ so that $DF(T_xE)$ is of dimension $l+1$ and
$$
DF(T_xE)+\textrm{Span}\{D^{\alpha_j}F(x)(v_j^{|\alpha_j|}): 1\leq j\leq l'+1-l\}
$$
is of dimension $l'+2$. Note that $l'+1-l\le l$. If we consider an $l$-plane $\mathcal L\subset E$ passing through $x$ such that $v_1$, $v_2$, $\cdots$, $v_{l'+1-l}$ are all contained in $T_x\mathcal L$, then the linear span of the image of $\mathcal L$ by $F$ contains
$$
DF(T_x\mathcal L)+\textrm{Span}\{D^{\alpha_j}F(x)(v_j^{|\alpha_j|}): 1\leq j\leq l'+1-l\},
$$
which is of dimension $l'+1$. Thus the image of the $l$-plane $\mathcal L$ is not contained in any $ l'$-plane, a contradiction. Therefore, $F$ maps $(l+1)$-planes into $(l'+1)$-planes.

\textbf{(Case II: $q< l+1$)} We can choose a point $x\in\mathbb P^m$ at which the rank of the Jacobian $DF$ is constantly $q$ in a neighborhood $U\ni x$ and such that there is a neighborhood $U'$ at $F(x)$ in which the image of $F(U)$ is a $q$-dimensional complex manifold in $U'$. Then, as $q\leq l$, for a general $l$-plane $\mathcal L$, we have $F(\mathcal L\cap U)=F(U)$, which by hypotheses, is contained in an $l'$-plane and thus it follows trivially that $F$ maps $(l+1)$-planes into $(l'+1)$-planes.

Combining \textbf{Case I} and \textbf{Case II}, we conclude that $F$ maps $(l+1)$-planes to $(l'+1)$-planes and thus by induction $F$ maps $(l+k)$-planes into $(l'+k)$-planes for $0\le k\le \min\{m-l,m'-l'\}$. The proof for part (i) is complete.

To prove part (ii), let $d$ be the dimension of the linear span of  the image of $\mathcal L_1\cap \mathcal L_2$ for a general pair of $l$-planes $\mathcal L_1$ and $\mathcal L_2$ such that intersection $\mathcal L_1\cap \mathcal L_2$ is an $(l-1)$-plane.  Thus, $d\leq l'$.

If $d=l'$, then the images of $\mathcal L_1$ and $\mathcal L_2$ must be contained in the same $l'$-plane. Since $\mathcal L_1$ and $\mathcal L_2$ are chosen generally, so $F$ must map every $l$-plane into this $l'$-plane. Therefore, the image of $F$ must be contained in this $l'$-plane.

If $d<l'$, then it means $F$ maps $\mathcal L_1\cap \mathcal L_2$ into an $(l'-1)$-plane. Since $\mathcal L_1, \mathcal L_2$ are chosen generally, it follows that $F$ maps $(l-1)$-planes to $(l'-1)$-planes.

Inductively, we deduce that either $F$ maps $(l-l'+1)$-planes to 1-planes (lines) or the image of $F$ is contained in an $k$-plane, $1\le k\le l'$.

If $l'=l$, then either $F$ maps lines to lines or the image of $F$
is contained in an $l'$-plane. Thus $F$ is linear or the image of
$F$ is contained in an $l'$-plane.

If $l'< l$, then either $F$ maps $(l-l')$-planes to 0-planes (points) or the image of
 $F$ is contained in an $l'$-plane. The first possibility implies that the image of
  $F$ is a single point. Thus, the image of $F$ must be contained in an $l'$-plane in
   any case.
 \end{proof}

\noindent\textbf{Acknowledgement.}
The author would like to thank Prof. Xiaojun Huang and Sui-Chung Ng for their  valuable comments on the first draft.
The author was partially supported by Institute of Marine Equipment, Shanghai Jiao Tong University and Shanghai Science and Technology Plan projects (No. 21JC1401900).

\end{document}